\documentclass[11pt,a4paper]{amsart}
\usepackage[T1]{fontenc}
\usepackage[utf8]{inputenc}
\usepackage{a4wide}
\usepackage{amsmath,amssymb,amsthm}
\usepackage[mathcal]{eucal}
\usepackage{tikz-cd}
\usepackage{url}
\usepackage[hidelinks]{hyperref}

\newcommand{\fn}{\operatorname{FN}}
\newcommand{\sfn}{\operatorname{SFN}}
\newcommand{\fni}{\operatorname{FNI}}
\newcommand{\fns}{\operatorname{FNS}}
\newcommand{\Fr}{\operatorname{Fr}}
\newcommand{\Clop}{\operatorname{Clop}}
\newcommand{\RO}{\operatorname{RO}}

\newcommand{\reg}{\leqslant_{\mathrm{reg}}}
\newcommand{\rc}{\leqslant_{\mathrm{rc}}}

\newcommand{\pos}[1]{{#1}^{+}}

\newcommand{\bfA}{ A}
\newcommand{\bfB}{ B}
\newcommand{\bfC}{ C}
\newcommand{\bfD}{ D}
\newcommand{\bfE}{ E}
\newcommand{\bfF}{ F}

\newcommand{\calC}{\mathcal C}

\newcommand{\calR}{\mathcal R}
\newcommand{\calS}{\mathcal S}
\newcommand{\calT}{\mathcal T}
\newcommand{\calX}{\mathcal X}

\newcommand{\At}{\operatorname{At}}
\newcommand{\Lim}{\operatorname{Lim}}

\newcommand{\ortho}{\perp}

\newtheorem{theorem}{Theorem}[section]
\newtheorem{proposition}[theorem]{Proposition}
\newtheorem{lemma}[theorem]{Lemma}
\newtheorem{corollary}[theorem]{Corollary}

\newtheorem{problem}[theorem]{Problem}
\theoremstyle{definition}
\newtheorem{definition}[theorem]{Definition}
\newtheorem{remark}[theorem]{Remark}

\newtheorem*{theoremA}{Theorem A}
\newtheorem*{theoremB}{Theorem B}
\newtheorem*{theoremC}{Theorem C}

\title[Coherent Freese--Nation separations]{Characterising projective and weakly projective Boolean algebras by coherent Freese--Nation separations}

\author[T. Kania]{Tomasz Kania}
\address[T.~Kania]{Mathematical Institute\\Czech Academy of Sciences\\\v Zitn\'a 25 \\115 67 Praha 1\\Czech Republic  and  Institute of Mathematics and Computer Science\\ Jagiellonian University\\ {\L}ojasiewicza 6, 30-348 Krak\'{o}w, Poland}
\email{kania@math.cas.cz, tomasz.marcin.kania@gmail.com}
\thanks{RVO: 67985840.}

\author[A. Kucharski]{Andrzej Kucharski}
\address[A.~Kucharski]{University of Silesia in Katowice\\Bankowa 14\\40-007 Katowice, Poland}
\email{andrzej.kucharski@us.edu.pl}

\author[S. Turek]{S\l awomir Turek}
\address[S.~Turek]{Institute of Mathematics\\Cardinal Stefan Wyszy\'nski University in Warsaw\\W\'oycickiego 1/3\\01-938 Warszawa, Poland}
\email{s.turek@uksw.edu.pl}

\subjclass[2020]{Primary 06E05; Secondary 06E10, 03E05, 54C55}
\keywords{Boolean algebra, projective Boolean algebra, weakly projective Boolean algebra, Freese--Nation property, strong Freese--Nation property, relatively complete subalgebra, Cohen algebra}
\date{14 July 2026}

\hypersetup{
  pdftitle={Characterising projective and weakly projective Boolean algebras by coherent Freese--Nation separations},
  pdfauthor={Tomasz Kania, Andrzej Kucharski, and Slawomir Turek},
  pdfsubject={Coherent finite-separation characterisations of projective and weakly projective Boolean algebras},
  pdfkeywords={Boolean algebra, projective Boolean algebra, weakly projective Boolean algebra, Freese--Nation property, strong Freese--Nation property, Cohen algebra}
}

\begin{document}

\begin{abstract}
We isolate a coherent finite-separation strengthening of the Freese--Nation property and use it to characterise projective and weakly projective Boolean algebras.  A Boolean algebra is projective if and only if it has a meet-closed decomposition base carrying such a coherent finite-separation map; it is weakly projective if and only if it has a meet-closed $\pi$-base carrying one. The global form of the same property is much stronger: the positive cone of a Boolean algebra carries a coherent finite-separation map if and only if the algebra is countable.  Thus the projective and weakly projective characterisations are intrinsically local and cannot be strengthened by requiring the whole algebra to carry the coherent map except in the countable case.
\end{abstract}

\maketitle

\section{Introduction}

A Boolean algebra $\bfA$ is \emph{projective} if, for every epimorphism $g\colon\bfC\twoheadrightarrow\bfB$ of Boolean algebras and every homomorphism $h\colon\bfA\to\bfB$, there is a homomorphism $k\colon\bfA\to\bfC$ such that $g\circ k=h$:
\[
\begin{tikzcd}[column sep=3.2em,row sep=2.6em]
& \bfC \arrow[d,two heads,"g"] \\
\bfA \arrow[ur,dashed,"k"] \arrow[r,"h"'] & \bfB .
\end{tikzcd}
\]
Halmos proved that every countable Boolean algebra is projective \cite{Halmos}.  In the uncountable case, projectivity is governed by the classical theory of relatively complete subalgebras, retracts of free Boolean algebras, and additive skeletons.  Equivalently, on the Stone side one is studying zero-dimensional compact spaces obtained as retracts of Cantor cubes, or, in the language of Shchepin and Haydon, the zero-dimensional part of the theory of Dugundji and AE$(0)$ compacta \cite{Haydon,Shchepin}.  These structural descriptions are exact and powerful, but they are global: they describe a whole transfinite resolution of the algebra.

The Freese--Nation property gives a quite different kind of information.  It asks for finite interpolation data attached to individual elements.  In Boolean algebras this is equivalent to a finite separation principle: two disjoint elements can be separated by two disjoint members of a finite set assigned to both of them.  This makes the Freese--Nation property a natural finite-combinatorial shadow of projectivity.  Freese and Nation proved such finite interpolation for projective lattices \cite{FreeseNation}; Heindorf and Shapiro later showed that, for Boolean algebras, the ordinary Freese--Nation property is controlled by clubs of countable relatively complete subalgebras and is weaker than projectivity in general \cite{HeindorfShapiro}.

The closest established strengthening is the \emph{strong Freese--Nation property} $(\sfn)$.  Two subalgebras commute when any two comparable elements, one from each subalgebra, admit an interpolant in their intersection; a Boolean algebra has $(\sfn)$ when it has a cofinal family of finite subalgebras which commute pairwise.  Heindorf and Shapiro proved
\[
        \text{projective}\quad\Longrightarrow\quad(\sfn)
        \quad\Longrightarrow\quad(\fn).
\]
Neither implication reverses in general: there are uncountable $(\sfn)$ algebras which are not projective, while Milovich constructed an algebra with $(\fn)$ but without $(\sfn)$ \cite{HeindorfShapiro,MilovichSFN}.  Milovich also gave game-theoretic and locally finite characterisations of projective Boolean algebras \cite{MilovichTeam}.  Those descriptions organise compatible finite subalgebras or finite local systems.  The condition introduced here is of a different form: it is carried by a single elementwise map on a decomposition base, and its compatibility is expressed directly through non-zero meets.  Its equivalence with projectivity therefore locates it strictly above $(\sfn)$ in general, while providing a particularly economical local certificate.

The point of this paper is that a single extra coherence requirement recovers projectivity exactly.  We require the finite data to respect non-zero intersections: if $a\wedge b>0$, then the finite set assigned to $a\wedge b$ must be contained in the family of non-zero finite meets generated by the finite sets assigned to $a$ and to $b$.  This is a local version of compatibility with refinement of basic clopen sets.  It is stronger than ordinary Freese--Nation interpolation, but still finite and local: the map is required only on a suitable base, not on the whole algebra.

This distinction between local and global is essential.  Ordinary Freese--Nation maps may live on large uncountable algebras.  By contrast, the global coherent version introduced here forces countability.  Thus the correct finite-combinatorial replacement for a projective skeleton is not a global coherent map on $\pos{\bfA}$, but a coherent map on a meet-closed decomposition base.  The resulting condition is useful because it turns the transfinite skeleton into a checkable finite rule on the pieces from which the algebra is assembled.  It also suggests new invariants, such as possible uniform bounds on the sizes of the witnessing finite sets.

Here is the first main result.

\begin{theoremA}
For a Boolean algebra $\bfA$, the following are equivalent.
\begin{enumerate}
\item $\bfA$ is projective.
\item $\bfA$ has a base $X\subseteq \pos{\bfA}$, closed under non-zero meets, which carries a coherent finite-separation map.
\end{enumerate}
\end{theoremA}

The weakly projective analogue is obtained by replacing bases by $\pi$-bases.  Recall that weak projectivity is the co-absolute, or completion-invariant, version of projectivity: a Boolean algebra is weakly projective precisely when it is co-complete with a projective Boolean algebra, equivalently when it has a dense projective subalgebra.

\begin{theoremB}
For a Boolean algebra $\bfA$, the following are equivalent.
\begin{enumerate}
\item $\bfA$ is weakly projective.
\item $\bfA$ has a $\pi$-base $X\subseteq \pos{\bfA}$, closed under non-zero meets, which carries a coherent finite-separation map.
\end{enumerate}
\end{theoremB}

Theorem~C explains why the base formulation cannot be replaced by a global one.

\begin{theoremC}
For a Boolean algebra $\bfA$, the positive cone $\pos{\bfA}$ carries a coherent finite-separation map if and only if $\bfA$ is countable.
\end{theoremC}

Theorem~C is the obstruction which makes the local formulation sharp.  Its uncountable part first uses Theorem~A to reduce to a subalgebra of a free Boolean algebra.  A $\Delta$-system and thinning argument then produces many elements with a common finite root, pairwise disjoint private supports, and one fixed proper Boolean coefficient.  Coherence forces the complement of a finite meet to be represented as a meet of finitely many values of the separation map; finite-support independence in the free algebra gives the contradiction.

The paper is organised as follows.  Section~\ref{sec:preliminaries} fixes notation, recalls the precise structural theorems used later, and proves the elementary closure lemmas for finite separation maps.  Section~\ref{sec:projective} proves Theorem~A; the proof of the projective-to-base implication is given with full successor and limit-stage bookkeeping.  Section~\ref{sec:weakly-projective} proves Theorem~B and the Cohen-algebra corollary.  Section~\ref{sec:global} proves Theorem~C.  The final section gives the Stone-dual form of the main results, records further questions, and uses the orthomodular boundary case to isolate the genuinely Boolean roles of distributivity, clopen decomposition, and finite-support independence.

\section{Preliminaries}\label{sec:preliminaries}

All Boolean algebras are written additively and multiplicatively as $(\bfA,\vee,\wedge,-,0,1)$.  We write
\[
        \pos{\bfA}=\bfA\setminus\{0\}
\] for the positive cone of \(\bfA\).
If $F$ is a finite subset of a Boolean algebra, put
\[
        F^{\wedge}=\{\bigwedge G: \emptyset\neq G\subseteq F\}\setminus\{0\}.
\]
Thus $F^{\wedge}$ is finite and is closed under non-zero finite meets.

A subset $X\subseteq\pos{\bfA}$ is a \emph{base} of $\bfA$ if every non-zero element of $\bfA$ is a finite join of pairwise disjoint members of $X$.  We also call such an $X$ a \emph{decomposition base}, in order to emphasise that the members occurring in each decomposition are disjoint, not that the whole family $X$ is pairwise disjoint.  Throughout the paper, ``base'' has this meaning.  A set $X\subseteq\pos{\bfA}$ is a \emph{$\pi$-base} if for every $a\in\pos{\bfA}$ there is $x\in X$ with $x\leqslant a$.  We shall always state explicitly when a base or $\pi$-base is closed under non-zero meets.

For $S\subseteq\bfA$, we denote by $\langle S\rangle$ the subalgebra of $\bfA$ generated by $S$.

\begin{definition}\label{def:fn}
Let $(P,\leqslant)$ be a partially ordered set.  The poset $P$ has the \emph{Freese--Nation property}, abbreviated $(\fn)$, if for every $p\in P$ there are finite sets
\[
        U(p)\subseteq\{r\in P:p\leqslant r\},
        \qquad
        L(p)\subseteq\{r\in P:r\leqslant p\},
\]
such that, whenever $p\leqslant q$, one has $U(p)\cap L(q)\neq\emptyset$.

Equivalently, $P$ has the \emph{finite-interpolation form} $(\fni)$ if there is a map
\[
        f\colon P\longrightarrow [P]^{<\omega}
\]
such that, whenever $p\leqslant q$, some $r\in f(p)\cap f(q)$ satisfies
\[
        p\leqslant r\leqslant q .
\]
A Boolean algebra has $(\fn)$, respectively $(\fni)$, if its underlying ordered set has that property.
\end{definition}

\begin{definition}\label{def:fns-star}
Let $X\subseteq\pos{\bfA}$ be closed under non-zero meets.  A map
\[
        s\colon X\longrightarrow [X]^{<\omega}
\]
is a \emph{finite-separation map}, or \emph{witnesses the finite-separation property} $(\fns)$ on $X$, if for every $a,b\in X$ with $a\wedge b=0$ there are $c,d\in s(a)\cap s(b)$ such that
\[
        c\wedge d=0,\qquad a\leqslant c,\qquad b\leqslant d .
\]
We say that $s$ is \emph{coherent}, or that $s$ witnesses $(\fns)^*$ on $X$, if, in addition,
\begin{equation}\label{eq:star}
        s(a\wedge b)\subseteq (s(a)\cup s(b))^{\wedge}
\end{equation}
whenever $a,b\in X$ and $a\wedge b>0$.
When $X=\pos{\bfA}$ we say that $\bfA$ has the \emph{global} $(\fns)$ property or the \emph{global} $(\fns)^*$ property, respectively.  In this paper a \emph{finite-separation map} means a witness to $(\fns)$; when the extra condition \eqref{eq:star} is included we say \emph{coherent finite-separation map}.  Accordingly, \emph{coherent finite-separation property} is the verbal name for $(\fns)^*$.
\end{definition}

For Boolean algebras the global finite-separation property is just the Freese--Nation property in separation form.  We record the elementary equivalence in order to fix the convention used throughout the paper.

\begin{proposition}[{\cite[Observation 2.2.2]{HeindorfShapiro}}]\label{prop:fn-fns-equivalence}
For a Boolean algebra $\bfA$, the properties $(\fn)$, $(\fni)$, and global $(\fns)$ are equivalent.
\end{proposition}

\begin{proof}
First, $(\fn)$ and $(\fni)$ are equivalent for every poset.  If $U,L$ witness $(\fn)$, then $f(p)=U(p)\cup L(p)$ witnesses $(\fni)$.  Conversely, if $f$ witnesses $(\fni)$, put
\[
        U(p)=f(p)\cap\{r:p\leqslant r\},
        \qquad
        L(p)=f(p)\cap\{r:r\leqslant p\}.
\]
Then $U,L$ witness $(\fn)$.

Suppose now that $f\colon\bfA\to[\bfA]^{<\omega}$ witnesses $(\fni)$.  For $a\in\pos{\bfA}$ put
\[
        s(a)=\pos{\bfA}\cap\bigl(f(a)\cup f(-a)
        \cup\{-u:u\in f(a)\cup f(-a)\}\bigr).
\]
Let $a,b\in\pos{\bfA}$ and $a\wedge b=0$.  Applying $(\fni)$ to $a\leqslant -b$, choose $u\in f(a)\cap f(-b)$ with $a\leqslant u\leqslant -b$.  Then $u$ and $-u$ both belong to $s(a)\cap s(b)$, they are disjoint, $a\leqslant u$, and $b\leqslant -u$.  Thus $s$ witnesses global $(\fns)$.

Conversely, suppose that $s\colon\pos{\bfA}\to[\pos{\bfA}]^{<\omega}$ witnesses global $(\fns)$.  Define
\[
        f(x)=\{0,1,x\}\cup
        \begin{cases}
        s(x),& x>0,\\
        \emptyset,& x=0,
        \end{cases}
        \cup
        \begin{cases}
        s(-x),& -x>0,\\
        \emptyset,& x=1 .
        \end{cases}
\]

for $x\in\bfA$.  If $a\leqslant b$ and either $a=0$ or $b=1$, then $0$ or $1$ gives the required interpolation.  Otherwise $a>0$ and $-b>0$.  Since $a\wedge -b=0$, finite separation gives disjoint $u,v\in s(a)\cap s(-b)$ with $a\leqslant u$ and $-b\leqslant v$.  Then $u\leqslant -v\leqslant b$, and $u\in f(a)\cap f(b)$.  Hence $f$ witnesses $(\fni)$.
\end{proof}

The following elementary observation will be used repeatedly without further comment.

\begin{lemma}\label{lem:finite-meet-star}
Let $s$ witness $(\fns)^*$ on $X$.  If $n\geqslant 1$, $x_0,\ldots,x_{n-1}\in X$, and $x=x_0\wedge\cdots\wedge x_{n-1}>0$, then
\[
        s(x)\subseteq \bigl(s(x_0)\cup\cdots\cup s(x_{n-1})\bigr)^{\wedge}.
\]
\end{lemma}

\begin{proof}
This is a direct induction on $n$.  The case $n=1$ is immediate, and the case $n=2$ is \eqref{eq:star}.  If the assertion holds for $n$ and $y=x_0\wedge\cdots\wedge x_{n-1}$, then
\[
 s(y\wedge x_n)\subseteq (s(y)\cup s(x_n))^{\wedge}
       \subseteq \bigl(s(x_0)\cup\cdots\cup s(x_n)\bigr)^{\wedge},
\]
because every member of $s(y)$ is already a non-zero finite meet of members of $s(x_0)\cup\cdots\cup s(x_{n-1})$.
\end{proof}

We shall use standard terminology for regular and relatively complete subalgebras; see Koppelberg \cite{KoppelbergVol1} or Heindorf--Shapiro \cite{HeindorfShapiro}.  A \emph{partition} is a maximal antichain.  A subalgebra $\bfA$ of $\bfB$ is \emph{regular}, written $\bfA\reg\bfB$, if every partition of $\bfA$ is a partition of $\bfB$.  It is \emph{relatively complete}, written $\bfA\rc\bfB$, if for every $b\in\bfB$ there is a least member of $\bfA$ above $b$.  When it exists, this least member is denoted by
\[
        q_{\bfA}^{\bfB}(b)=\bigwedge_{\bfA}\{a\in\bfA: b\leqslant a\}
\]
and we say that $b$ is \emph{$\bfA$-regular}.
We suppress the superscripts when no confusion is possible. The following properties of the map $q$ are well known, see
\cite{BlaszczykKucharskiTurek}, \cite{KoppelbergProjective}.

\begin{lemma}\label{lem:q-basic}
Let $\bfA\leqslant\bfB$ and $b,c\in\bfB$.  Suppose that $q(b)$ and $q(c)$ exist.
\begin{enumerate}
\item If $a\in\bfA$, then $q(a\wedge b)$ exists and equals $a\wedge q(b)$.
\item The element $q(b\vee c)$ exists and equals $q(b)\vee q(c)$.
\item If $0<a\leqslant q(b)$ and $a\in\bfA$, then $a\wedge b>0$.
\end{enumerate}
\end{lemma}

\begin{proof}
For (1), $a\wedge q(b)$ is an element of $\bfA$ above $a\wedge b$.  If $d\in\bfA$ and $a\wedge b\leqslant d$, then
$b\leqslant -a\vee d$, hence $q(b)\leqslant -a\vee d$, and so $a\wedge q(b)\leqslant d$.  This proves (1).  For (2), $q(b)\vee q(c)$ is plainly above $b\vee c$ and is below every member of $\bfA$ above $b\vee c$.  Finally, if $0<a\leqslant q(b)$ and $a\wedge b=0$, then $b\leqslant q(b)\wedge -a<q(b)$, contradicting the minimality of $q(b)$.
\end{proof}

\begin{lemma}\label{lem:rc-composition}
Let $\bfA\rc\bfB\rc\bfC$.  Then $\bfA\rc\bfC$ and
\[
        q_{\bfA}^{\bfC}=q_{\bfA}^{\bfB}\circ q_{\bfB}^{\bfC}.
\]
Consequently, in a continuous chain in which every successor inclusion is relatively complete, every earlier algebra is relatively complete in every later one; if $\gamma\leqslant\beta\leqslant\alpha$, then
\[
        q_\gamma^\alpha=q_\gamma^\beta\circ q_\beta^\alpha.
\]
In particular, if $b\in\bfA_\beta$ and $\beta\leqslant\eta\leqslant\alpha$, then $q_\eta^\alpha(b)=b$.
\end{lemma}

\begin{proof}
For $c\in\bfC$, the element $q_{\bfA}^{\bfB}(q_{\bfB}^{\bfC}(c))$ belongs to $\bfA$ and is above $c$.  If $a\in\bfA$ and $c\leqslant a$, then $q_{\bfB}^{\bfC}(c)\leqslant a$, whence $q_{\bfA}^{\bfB}(q_{\bfB}^{\bfC}(c))\leqslant a$.  This proves both relative completeness and the displayed identity.  The assertion for continuous chains follows by induction on the later ordinal; at a limit stage the element under consideration already belongs to some earlier member of the continuous union.  The final statement is immediate from the definition of least cover.
\end{proof}

\begin{lemma}\label{lem:regular-criterion}
For $\bfA\leqslant\bfB$, the following are equivalent.
\begin{enumerate}
\item $\bfA\reg\bfB$.
\item For every $b\in\pos{\bfB}$ there is $a\in\pos{\bfA}$ such that every $x\in\pos{\bfA}$ with $x\leqslant a$ satisfies $x\wedge b>0$.
\end{enumerate}
\end{lemma}

\begin{proof}
This is the standard partition criterion for regular subalgebras; see, for instance, \cite[Proposition 2.1]{BlaszczykKucharskiTurek} or \cite[Section 1]{KoppelbergVol1}.  Indeed, failure of (2) gives a maximal antichain in $\bfA$ disjoint from $b$, hence a partition of $\bfA$ which is not maximal in $\bfB$; conversely, if a partition of $\bfA$ is not maximal in $\bfB$, a non-zero element of $\bfB$ disjoint from all its members witnesses the failure of (2).
\end{proof}

The next lemma is the main mechanism by which finite separation maps produce regular or relatively complete subalgebras.

\begin{lemma}\label{lem:closed-subalgebra}
Let $X\subseteq\pos{\bfB}$ be a $\pi$-base closed under non-zero meets, and let $s\colon X\to[X]^{<\omega}$ be a finite-separation map.  Let $\bfA\leqslant\bfB$.  Suppose that there is a $\pi$-base $P\subseteq X\cap\bfA$ of $\bfA$ such that $s(p)\subseteq\bfA$ for every $p\in P$.  Then $\bfA\reg\bfB$.

If, in addition, $X$ is a base of $\bfB$ and $P$ is a base of $\bfA$, then $\bfA\rc\bfB$.
\end{lemma}

\begin{proof}
First fix $b\in X$.  Put
\[
        q_X(b)=\begin{cases}
        \bigwedge\{c\in s(b)\cap\bfA: b\leqslant c\},&\text{if this set is non-empty},\\
        1,&\text{otherwise.}
        \end{cases}
\]
Then $q_X(b)\in\bfA$ and $b\leqslant q_X(b)$.  We claim that $q_X(b)$ is the least member of $\bfA$ above $b$.
Let $a\in\bfA$ and $b\leqslant a$.  If $q_X(b)\wedge -a>0$, choose $p\in P$ such that $p\leqslant q_X(b)\wedge -a$.  Then $p\wedge b=0$, so finite separation gives disjoint $u,v\in s(p)\cap s(b)$ with $p\leqslant u$ and $b\leqslant v$.  Since $p\in P$, $s(p)\subseteq\bfA$, hence $v\in s(b)\cap\bfA$ and $b\leqslant v$.  Thus the finite set used in the first case of the definition of $q_X(b)$ is non-empty.  We are therefore in that first case, and $q_X(b)\leqslant v$.  Hence $p\leqslant q_X(b)\leqslant v$, contradicting $u\wedge v=0$ and $0<p\leqslant u$.  Thus $q_X(b)\leqslant a$, proving the claim.

To prove regularity, take any $b\in\pos{\bfB}$ and choose $x\in X$ with $x\leqslant b$.  The claim gives $q_X(x)\in\pos{\bfA}$ above $x$.  If $0<a\leqslant q_X(x)$ and $a\in\bfA$, then $a\wedge x>0$ by Lemma~\ref{lem:q-basic}(3), and hence $a\wedge b>0$.  Lemma~\ref{lem:regular-criterion} gives $\bfA\reg\bfB$.

Assume now that $X$ and $P$ are bases.  We have proved that every member of $X$ has a least cover in $\bfA$.  If $b\in\pos{\bfB}$, write $b=x_0\vee\cdots\vee x_{n-1}$ as a finite disjoint join of members of $X$ and put
\[
        q_{\bfA}^{\bfB}(b)=q_X(x_0)\vee\cdots\vee q_X(x_{n-1}),
        \qquad q_{\bfA}^{\bfB}(0)=0.
\]
This does not depend on the chosen decomposition: if $a\in\bfA$ and $b\leqslant a$, then $x_i\leqslant a$ for every $i<n$, so $q_X(x_i)\leqslant a$ for every $i<n$.  Hence the displayed join is the least member of $\bfA$ above $b$.  Thus the genuine projection $q_{\bfA}^{\bfB}$ exists on all of $\bfB$, and $\bfA\rc\bfB$.
\end{proof}

We shall also use two standard structural characterisations.  We spell out the terminology because the proofs below use the exact forms, not only the informal consequences.

\begin{definition}\label{def:skeleton-terminology}
Two Boolean algebras are \emph{co-complete} if their Boolean completions are isomorphic.  A subalgebra $\bfD\leqslant\bfA$ is \emph{dense} if every non-zero member of $\bfA$ contains a non-zero member of $\bfD$.  A set $\calC$ of countable subalgebras of $\bfA$ is \emph{closed unbounded}, or a \emph{club}, if it is cofinal in $[\bfA]^\omega$ and is closed under unions of increasing countable chains.

We reserve \emph{additive relatively complete skeleton} for the standard four-axiom notion of Heindorf--Shapiro \cite[Section~1.3]{HeindorfShapiro}.  Thus it is a family $\calS$ of subalgebras of $\bfA$ such that:
\begin{enumerate}
\item the union of every chain in $\calS$ belongs to $\calS$;
\item every subalgebra $\bfC\leqslant\bfA$ is contained in some $\bfD\in\calS$ with $|\bfC|=|\bfD|$;
\item every $\bfD\in\calS$ satisfies $\bfD\rc\bfA$;
\item $\langle\bigcup\calT\rangle\in\calS$ for every subfamily $\calT\subseteq\calS$.
\end{enumerate}
Replacing $\rc$ by $\reg$ gives the standard analogous notion of an \emph{additive regular skeleton}.

For the streamlined formulations used in Propositions~\ref{prop:base-to-projective} and~\ref{prop:pibase-to-weak}, we introduce separate terminology.  A \emph{countable rc-generating family} for $\bfA$ is a family $\calR$ of countable subalgebras such that $\bfA=\bigcup\calR$ and
\[
        \Bigl\langle\bigcup\calC\Bigr\rangle\rc\bfA
\]
for every subfamily $\calC\subseteq\calR$.  Replacing $\rc$ by $\reg$ gives a \emph{countable regular-generating family}.  These are equivalent countable generating-family formulations, not our definition of a skeleton.  A Boolean algebra with a club of countable regular subalgebras is called \emph{regularly filtered}; the corresponding relatively complete notion is \emph{rc-filtered}.
\end{definition}

\begin{theorem}[Koppelberg, Shchepin, Haydon]\label{thm:projective-known}
For a Boolean algebra $\bfA$, the following are equivalent.
\begin{enumerate}
\item $\bfA$ is projective.
\item $\bfA$ is a retract of a free Boolean algebra.
\item There are an ordinal $\kappa$ and a continuous chain $(\bfA_\alpha)_{\alpha\leqslant\kappa}$ of subalgebras such that
\[
        \bfA=\bfA_\kappa,\qquad \bfA_0=\{0,1\},
\]
and, for every $\alpha<\kappa$, $\bfA_\alpha\rc\bfA_{\alpha+1}$ and $\bfA_{\alpha+1}=\bfA_\alpha(e_\alpha)$ for some $e_\alpha\in\bfA_{\alpha+1}$.
\item There are an ordinal $\kappa$ and a continuous chain $(\bfA_\alpha)_{\alpha\leqslant\kappa}$ of subalgebras such that $\bfA=\bfA_\kappa$, $\bfA_0$ is countable, and, for every $\alpha<\kappa$, $\bfA_\alpha\rc\bfA_{\alpha+1}$ and $\bfA_{\alpha+1}$ is countably generated over $\bfA_\alpha$.
\item $\bfA$ has a countable rc-generating family.
\item $\bfA$ has an additive relatively complete skeleton in the standard sense of Definition~\ref{def:skeleton-terminology}.
\end{enumerate}
\end{theorem}

\begin{proof}
The retract formulation and the simple extension filtration are given in Koppelberg \cite[Chapter~20, in particular Theorem~2.8]{KoppelbergProjective}.  The countably generated filtration is Haydon's formulation \cite[pp.~23--31]{Haydon}.  The additive-skeleton and countable rc-generating-family formulations, together with their equivalence to projectivity, are recorded in Heindorf--Shapiro \cite[Theorem~1.3.2]{HeindorfShapiro}; the skeleton formulation originates with Shchepin \cite{Shchepin}.  These references use the same arbitrary-subfamily generated-union condition as in Definition~\ref{def:skeleton-terminology} and condition~(5).
\end{proof}

\begin{theorem}[Shapiro, Balcar--Jech--Zapletal, Koppelberg--Bandlow]\label{thm:weak-known}
For a Boolean algebra $\bfA$, the following are equivalent.
\begin{enumerate}
\item $\bfA$ is weakly projective, that is, co-complete with a projective Boolean algebra.
\item $\bfA$ has a dense projective subalgebra.
\item $\bfA$ has a countable regular-generating family.
\item $\bfA$ has a club $\calC$ of countable regular subalgebras such that $\langle\bfC_0\cup\bfC_1\rangle\in\calC$ whenever $\bfC_0,\bfC_1\in\calC$.
\item $\bfA$ has an additive regular skeleton in the standard sense of Definition~\ref{def:skeleton-terminology}.
\end{enumerate}
\end{theorem}

\begin{proof}
The equivalence of weak projectivity with the existence of a dense projective subalgebra is Shapiro's theorem; see Heindorf--Shapiro \cite[Theorem~5.3.1]{HeindorfShapiro} and Shapiro \cite{Shapiro}.  The remaining formulations are collected in Heindorf--Shapiro \cite[Theorem~5.3.9]{HeindorfShapiro}.  In that theorem the club formulation is attributed to Balcar--Jech--Zapletal \cite{BalcarJechZapletal}, while the countable regular-generating-family formulation is attributed to Koppelberg and Bandlow \cite{Bandlow}; see also Koppelberg \cite{KoppelbergCohen}.
\end{proof}

\section{Projective Boolean algebras}\label{sec:projective}

We begin with the easier direction of Theorem A: a coherent finite-separation map on a base yields the countable rc-generating family from Theorem~\ref{thm:projective-known}(5).

Let $X$ be closed under non-zero meets and let $s\colon X\to[X]^{<\omega}$.  For a countable set $S\subseteq X$, write $\operatorname{cl}_{s,\wedge}(S)$ for the least subset of $X$ which contains $S$, is closed under non-zero finite meets, and contains $s(p)$ for every one of its members $p$.  Explicitly, put $S_0=S$ and
\[
 S_{k+1}=S_k\cup\bigcup_{p\in S_k}s(p)
 \cup\left\{\bigwedge F:\emptyset\neq F\in[S_k]^{<\omega}\right\}\setminus\{0\},
\]
and take $\operatorname{cl}_{s,\wedge}(S)=\bigcup_{k<\omega}S_k$.  In particular, the least iterated closure of a countable set is countable.

\begin{proposition}\label{prop:base-to-projective}
Let $\bfA$ be a Boolean algebra.  Suppose that $X\subseteq\pos{\bfA}$ is a base closed under non-zero meets and that $s\colon X\to[X]^{<\omega}$ witnesses $(\fns)^*$ on $X$.  Then $\bfA$ is projective.
\end{proposition}

\begin{proof}
We shall build a countable rc-generating family.  First note the following closure fact.  If $D\in[\bfA]^\omega$, then there is a countable subalgebra $\bfD\leqslant\bfA$ containing $D$ and a base $P_\bfD\subseteq X\cap\bfD$ of $\bfD$ such that $s(p)\subseteq P_\bfD$ for all $p\in P_\bfD$.

Indeed, construct increasing countable sets $P_n\subseteq X$ and increasing countable subalgebras $\bfD_n$ recursively.  For each non-zero $d\in D$, choose a finite disjoint family of members of $X$ whose join is $d$; the element $0$, if present, is represented by the empty join.  Let $Q_0$ be the set of all pieces so chosen, put
\[
        P_0=\operatorname{cl}_{s,\wedge}(Q_0),
        \qquad \bfD_0=\langle P_0\rangle,
\]
so that $D\subseteq\bfD_0$.  Having constructed $P_n$ and $\bfD_n$, decompose every non-zero element of $\bfD_n$ as a finite disjoint join of members of $X$, and let $Q_{n+1}$ be the set of all pieces appearing in these decompositions.  Put
\[
        P_{n+1}=\operatorname{cl}_{s,\wedge}(P_n\cup Q_{n+1}),
        \qquad
        \bfD_{n+1}=\langle P_{n+1}\rangle.
\]
Thus $P_n\subseteq P_{n+1}$, and the operation $s$ is iterated until the finite $s$-sets of every newly introduced element are included.  Finally set
\[
        P_\bfD=\bigcup_{n<\omega}P_n,
        \qquad
        \bfD=\bigcup_{n<\omega}\bfD_n .
\]
Then $\bfD$ is a countable subalgebra containing $D$, every non-zero element of $\bfD$ is decomposed at a later stage into finitely many members of $P_\bfD$, and $s(p)\subseteq P_\bfD\subseteq\bfD$ for each $p\in P_\bfD$.  Hence $P_\bfD$ is a base of $\bfD$ with the required closure property.

Let $\calR$ be the family of all countable subalgebras $\bfD\leqslant\bfA$ for which such a base $P_\bfD$ exists.  By the preceding paragraph $\bfA=\bigcup\calR$.

Fix an arbitrary subfamily $\calC\subseteq\calR$ and put
\[
        \bfE=\Bigl\langle\bigcup_{\bfD\in\calC}\bfD\Bigr\rangle.
\]
If $\calC=\emptyset$, then $\bfE=\{0,1\}\rc\bfA$, so assume that $\calC\neq\emptyset$.  Let
\[
        P=\{p_0\wedge\cdots\wedge p_{n-1}: n\geqslant 1,
        \ p_i\in P_{\bfD_i},\ \bfD_i\in\calC\}\setminus\{0\}.
\]
Then $P\subseteq X\cap\bfE$.  We verify that $P$ is a base of $\bfE$.  Let $0<e\in\bfE$.  The element $e$ belongs to the subalgebra generated by finitely many members $\bfD_0,\ldots,\bfD_{m-1}$ of $\calC$.  Choose a finite Boolean subalgebra $\bfF_i\leqslant\bfD_i$ for each $i<m$ so that $e\in\langle\bfF_0\cup\cdots\cup\bfF_{m-1}\rangle$.  Every atom of each $\bfF_i$ is a finite disjoint join of elements of $P_{\bfD_i}$, because $P_{\bfD_i}$ is a base of $\bfD_i$.  Distributing the finite Boolean expression for $e$ over these finite decompositions gives a finite disjoint decomposition of $e$ into non-zero finite meets of members of the $P_{\bfD_i}$'s.  These finite meets are precisely members of $P$.  Hence $P$ is a base of $\bfE$.

If $p=p_0\wedge\cdots\wedge p_{n-1}\in P$, Lemma~\ref{lem:finite-meet-star} gives
\[
        s(p)\subseteq (s(p_0)\cup\cdots\cup s(p_{n-1}))^{\wedge}\subseteq\bfE,
\]
because each $s(p_i)$ is contained in the corresponding $\bfD_i$.  Lemma~\ref{lem:closed-subalgebra} therefore yields $\bfE\rc\bfA$.  The family $\calR$ is a countable rc-generating family, so Theorem~\ref{thm:projective-known}(5) implies that $\bfA$ is projective.
\end{proof}

The converse uses the simple extension relatively complete construction of projective algebras.  We first isolate the precise jump-set fact used at limit stages.  This is the form of Heindorf--Shapiro \cite[Lemmas~2.1.1 and~2.1.3]{HeindorfShapiro} after translating their projections into our maps $q_\gamma^\beta$.

\begin{lemma}[Jump sets]\label{lem:jump-sets}
Let $(\bfA_\xi)_{\xi\leqslant\alpha}$ be a continuous increasing chain of Boolean algebras such that $\bfA_\xi\rc\bfA_\eta$ whenever $\xi\leqslant\eta\leqslant\alpha$.  For $u\in\bfA_\alpha$, put
\[
        d_\alpha(u)=\{\xi<\alpha:q_\xi^\alpha(u)>q_{\xi+1}^\alpha(u)\}.
\]
Then:
\begin{enumerate}
\item $d_\alpha(u)$ is finite.  If $\alpha$ is a limit ordinal and $u\in\bfA_\beta$ for some $\beta<\alpha$, then $d_\alpha(u)\subseteq\beta$.
\item For every $\lambda\leqslant\alpha$,
\[
        d_\lambda(q_\lambda^\alpha(u))=d_\alpha(u)\cap\lambda.
\]
\item If $u,v\in\bfA_\alpha$ and $u\wedge v=0$, then either
\[
        q_0^\alpha(u)\wedge q_0^\alpha(v)=0,
\]
or there is $\xi\in d_\alpha(u)\cap d_\alpha(v)$ such that
\[
        q_{\xi+1}^\alpha(u)\wedge q_{\xi+1}^\alpha(v)=0.
\]
\end{enumerate}
In particular, the lemma applies to the continuous relatively complete chain with simple extension on successors in Theorem~\ref{thm:projective-known}(3).
\end{lemma}

\begin{proof}
We prove finiteness by induction on $\alpha$.  It is trivial for $\alpha=0$.  If $\alpha=\beta+1$ and $u\in\bfA_{\beta+1}$, put $u'=q_\beta^{\beta+1}(u)$.  The composition rule gives
\[
        d_{\beta+1}(u)\cap\beta=d_\beta(u'),
\]
so $d_{\beta+1}(u)$ is obtained from the finite set $d_\beta(u')$ by possibly adding $\beta$.  If $\alpha$ is a non-zero limit, continuity gives $u\in\bfA_\beta$ for some $\beta<\alpha$.  Then $q_\eta^\alpha(u)=u$ for every $\eta\in[\beta,\alpha]$, and the composition rule gives $d_\alpha(u)=d_\beta(u)$.  This proves (1).

For (2), if $\xi<\lambda$, Lemma~\ref{lem:rc-composition} yields
\[
 q_\xi^\lambda(q_\lambda^\alpha(u))=q_\xi^\alpha(u),
 \qquad
 q_{\xi+1}^\lambda(q_\lambda^\alpha(u))=q_{\xi+1}^\alpha(u),
\]
which is exactly the asserted restriction identity.

For (3), let $\delta\leqslant\alpha$ be the least ordinal such that
\[
        q_\delta^\alpha(u)\wedge q_\delta^\alpha(v)=0.
\]
Such a $\delta$ exists because the meet is zero at stage $\alpha$.  If $\delta=0$, the first alternative holds.  A non-zero limit $\delta$ is impossible: by continuity, both $q_\delta^\alpha(u)$ and $q_\delta^\alpha(v)$ belong to one common earlier algebra $\bfA_\eta$, and then the same zero meet already occurs at $\eta<\delta$.  Hence $\delta=\xi+1$ for some $\xi<\alpha$.

Suppose, for example, that $\xi\notin d_\alpha(u)$.  Then
\[
        a=q_\xi^\alpha(u)=q_{\xi+1}^\alpha(u)\in\bfA_\xi.
\]
Since $a\wedge q_{\xi+1}^\alpha(v)=0$, we have $q_{\xi+1}^\alpha(v)\leqslant-a$.  By the minimality of the $\bfA_\xi$-cover,
\[
        q_\xi^\alpha(v)
        =q_\xi^{\xi+1}(q_{\xi+1}^\alpha(v))\leqslant-a,
\]
contradicting the minimality of $\delta$.  Thus $\xi\in d_\alpha(u)$, and the symmetric argument gives $\xi\in d_\alpha(v)$.  The zero meet at stage $\xi+1$ gives the second alternative.
\end{proof}

The proof below is included in some detail because the witnessing finite sets must remain inside the chosen base.

\begin{proposition}\label{prop:projective-to-base}
Every projective Boolean algebra has a base closed under non-zero meets which carries a coherent finite-separation map.
\end{proposition}

\begin{proof}
Let $(\bfA_\alpha)_{\alpha\leqslant\kappa}$ be a continuous chain as in Theorem~\ref{thm:projective-known}(3), meaning that for every \(\alpha<\kappa\), $\bfA_\alpha\rc\bfA_{\alpha+1}$ and $\bfA_{\alpha+1}=\bfA_\alpha(e_\alpha)$ for some $e_\alpha\in\bfA_{\alpha+1}$.  By Lemma~\ref{lem:rc-composition}, $\bfA_\gamma\rc\bfA_\beta$ whenever $\gamma\leqslant\beta\leqslant\kappa$; write $q_\gamma^\beta$ for the corresponding least-cover map.

We construct, by induction on $\alpha\leqslant\kappa$, a base $X_\alpha\subseteq\pos{\bfA_\alpha}$ and a map $s_\alpha\colon X_\alpha\to[X_\alpha]^{<\omega}$.  At every successor stage $\alpha+1$ we shall also choose a finite set $T_\alpha\subseteq X_{\alpha+1}$.  The induction requirements are the following.
\begin{enumerate}
\item[(I1)] $X_\alpha$ is closed under non-zero meets.
\item[(I2)] $X_\gamma\subseteq X_\alpha$ whenever $\gamma<\alpha$.
\item[(I3)] If $a\in X_\alpha$, $\gamma<\alpha$,  then $q_\gamma^\alpha(a)\in X_\gamma$.
\item[(I4)] If $a,b\in X_\alpha$, $a\wedge b>0$, and $\gamma<\alpha$, then
\[
        q_\gamma^\alpha(a\wedge b)=q_\gamma^\alpha(a)\wedge q_\gamma^\alpha(b).
\]
\item[(I5)] $s_\alpha$ witnesses $(\fns)^*$ on $X_\alpha$.
\end{enumerate}
For $\bfA_0=\{0,1\}$ take $X_0=\{1\}$ and $s_0(1)=\{1\}$.

Assume first that $X_\alpha$ and $s_\alpha$ have been constructed.  Put
\[
        Q=q_\alpha^{\alpha+1},\qquad e=e_\alpha,
        \qquad e^0=e,\quad e^1=-e,
\]
and
\[
        a_i=Q(e^i)\qquad(i=0,1).
\]
By Lemma~\ref{lem:q-basic}(2), $a_0\vee a_1=1$.  The three elements
\[
        a_0\wedge -a_1,\quad a_0\wedge a_1,
        \quad a_1\wedge -a_0
\]
form a finite partition of $1$ after deleting zeros.  Since $X_\alpha$ is a base, choose a finite pairwise disjoint set $R_\alpha\subseteq X_\alpha$ such that $\bigvee R_\alpha=1$ and each $r\in R_\alpha$ is contained in one of these three pieces.  Define
\[
        X_{\alpha+1}=X_\alpha\cup\left(
        \{x\wedge r\wedge e^i: x\in X_\alpha,
        \ r\in R_\alpha,
        \ i<2\}\setminus\{0\}\right)\cup T_\alpha,
\]
where
\[
        T_\alpha=\{r\wedge e^i:r\in R_\alpha,
        \ i<2\}\setminus\{0\}.
\]
The set $X_{\alpha+1}$ is a base of $\bfA_{\alpha+1}=\bfA_\alpha(e)$.  Indeed, every element of $\bfA_\alpha(e)$ is a finite disjoint join of elements of the form $b\wedge e^i$ with $b\in\bfA_\alpha$; decompose $b$ into a finite disjoint join of elements of $X_\alpha$ and then refine by the finite partition $R_\alpha$.  Every member of $X_{\alpha+1}\setminus X_\alpha$ can be written as
\[
        x\wedge r\wedge e^i
        \qquad\text{with }x\in X_\alpha\cup\{1\},\ r\in R_\alpha,\ i<2.
\]
Using this representation, closure under non-zero meets is immediate: the meet of two old elements is old, the meet of an old element with a new element is either zero or has the same displayed form, and the meet of two new elements is either zero or again has that form.

Let $x\wedge r\wedge e^i>0$, where $x\in X_\alpha\cup\{1\},\ r\in R_\alpha,\ i<2$.  Since $e^i\leqslant a_i=Q(e^i)$, the element $r$ cannot be contained in the one-sided partition piece $a_{1-i}\wedge-a_i$; hence $r\leqslant a_i$.  Lemma~\ref{lem:q-basic}(1) therefore gives
\begin{equation}\label{eq:successor-q}
        Q(x\wedge r\wedge e^i)=x\wedge r .
\end{equation}
For old elements $x\in X_\alpha$ we of course have $Q(x)=x$.  This verifies (I3) for $\gamma=\alpha$.  If $\gamma<\alpha$, then Lemma~\ref{lem:rc-composition} gives
\[
        q_\gamma^{\alpha+1}(u)=q_\gamma^\alpha(Q(u))
        \qquad(u\in X_{\alpha+1}),
\]
and (I3) follows from the induction hypothesis at stage $\alpha$.

We next verify (I4) at the successor stage.  First we prove
\begin{equation}\label{eq:successor-meet-q}
        Q(u\wedge v)=Q(u)\wedge Q(v)
\end{equation}
whenever $u,v\in X_{\alpha+1}$ and $u\wedge v>0$.  If both elements are old, this is trivial.  If $u\in X_\alpha$ and $v=x\wedge r\wedge e^i$ is new, then Lemma~\ref{lem:q-basic}(1) gives
\[
        Q(u\wedge v)=u\wedge Q(v)=Q(u)\wedge Q(v).
\]
If both elements are new, write $u=x\wedge r\wedge e^i$ and $v=y\wedge r'\wedge e^j$, where $x,y\in X_\alpha\cup\{1\}$.  Then non-zero intersection forces $r=r'$ and $i=j$, and \eqref{eq:successor-q} gives
\[
        Q(u\wedge v)=x\wedge y\wedge r=Q(u)\wedge Q(v).
\]
Thus \eqref{eq:successor-meet-q} holds.  For $\gamma<\alpha$, combine \eqref{eq:successor-meet-q}, Lemma~\ref{lem:rc-composition}, and the induction hypothesis:
\[
q_\gamma^{\alpha+1}(u\wedge v)
 =q_\gamma^\alpha(Q(u\wedge v))
  =q_\gamma^\alpha(Q(u)\wedge Q(v))  
 =q_\gamma^\alpha(Q(u))\wedge q_\gamma^\alpha(Q(v))
  =q_\gamma^{\alpha+1}(u)\wedge q_\gamma^{\alpha+1}(v).
\]
This proves (I4).

Define
\begin{equation}\label{eq:successor-s}
        s_{\alpha+1}(u)=T_\alpha\cup s_\alpha(Q(u))
        \qquad(u\in X_{\alpha+1}).
\end{equation}
The right-hand side is a finite subset of $X_{\alpha+1}$.  We verify finite separation.  Let $u,v\in X_{\alpha+1}$ and $u\wedge v=0$.  If $Q(u)\wedge Q(v)=0$, let $c,d\in s_\alpha(Q(u))\cap s_\alpha(Q(v))$ be the separating pair given by the induction hypothesis.  Then $c,d\in s_{\alpha+1}(u)\cap s_{\alpha+1}(v)$, they are disjoint, and they dominate $u$ and $v$ because $u\leqslant Q(u)$ and $v\leqslant Q(v)$.

It remains to consider the case $Q(u)\wedge Q(v)>0$.  The two elements cannot both be old, and they cannot consist of one old element and one new element.  Indeed, if $u\in X_\alpha$ and $v$ is new, then $0<u\wedge Q(v)\leqslant Q(v)$; Lemma~\ref{lem:q-basic}(3), applied to $\bfA_\alpha\rc\bfA_{\alpha+1}$, gives $(u\wedge Q(v))\wedge v>0$, contradicting $u\wedge v=0$.  Hence both are new.  Write
\[
        u=x\wedge r\wedge e^i,
        \qquad
        v=y\wedge r'\wedge e^j,
        \qquad x,y\in X_\alpha\cup\{1\}.
\]
Since $Q(u)\wedge Q(v)>0$, the disjointness of $R_\alpha$ gives $r=r'$.  Since $u\wedge v=0$, we must have $i\neq j$.  Indeed, if $i=j$, then
\[
        0<Q(u)\wedge Q(v)=x\wedge y\wedge r\leqslant Q(e^i).
\]
Lemma~\ref{lem:q-basic}(3) would give
\[
        u\wedge v=x\wedge y\wedge r\wedge e^i>0,
\]
contrary to the assumption.  The two elements $r\wedge e^i$ and $r\wedge e^j$ are disjoint members of $T_\alpha\subseteq s_{\alpha+1}(u)\cap s_{\alpha+1}(v)$ and they dominate $u$ and $v$, respectively.  Thus $s_{\alpha+1}$ witnesses $(\fns)$.

Finally let $u,v\in X_{\alpha+1}$ and $u\wedge v>0$.  By \eqref{eq:successor-meet-q},
\[
        Q(u\wedge v)=Q(u)\wedge Q(v).
\]
Using the coherence of $s_\alpha$ and \eqref{eq:successor-s},
\[
 s_{\alpha+1}(u\wedge v)
=T_\alpha\cup s_\alpha(Q(u)\wedge Q(v))
\subseteq T_\alpha\cup\bigl(s_\alpha(Q(u))\cup s_\alpha(Q(v))\bigr)^{\wedge}
\subseteq \bigl(s_{\alpha+1}(u)\cup s_{\alpha+1}(v)\bigr)^{\wedge}.
\]
So (I5) holds at $\alpha+1$.

Now let $\alpha$ be a non-zero limit ordinal and suppose that the construction has been carried out below $\alpha$.  Put
\[
        X_\alpha=\bigcup_{\xi<\alpha}X_\xi .
\]
This is a base of $\bfA_\alpha$, because every element of the continuous union $\bfA_\alpha=\bigcup_{\xi<\alpha}\bfA_\xi$ already belongs to some earlier algebra and is decomposed there.  It is closed under non-zero meets: if $u\in X_\beta$ and $v\in X_\delta$, take $\eta<\alpha$ with $\beta,\delta\leqslant\eta$, and use (I1) at stage $\eta$.

The projection invariants are also inherited explicitly.  If $u\in X_\beta$ for some $\beta<\alpha$, then
\begin{equation}\label{eq:old-at-limit}
        q_\eta^\alpha(u)=u\quad(\beta\leqslant\eta<\alpha),
        \qquad
        q_\gamma^\alpha(u)=q_\gamma^\beta(u)\quad(\gamma<\beta).
\end{equation}
The first equality is because $u\in\bfA_\eta$; the second follows from Lemma~\ref{lem:rc-composition}.  Hence (I3) at $\alpha$ follows from (I3) at $\beta$.  If $u,v\in X_\beta$ and $u\wedge v>0$, then \eqref{eq:old-at-limit} and (I4) at stage $\beta$ give
\[
        q_\gamma^\alpha(u\wedge v)
        =q_\gamma^\alpha(u)\wedge q_\gamma^\alpha(v)
        \qquad(\gamma<\alpha),
\]
where the case $\gamma\geqslant\beta$ is trivial.  Thus (I4) is preserved at the limit stage.

For $u\in X_\alpha$, use the jump set $d_\alpha(u)$ from Lemma~\ref{lem:jump-sets}.  It is finite, and for $\lambda<\alpha$ the restriction identity reads
\begin{equation}\label{eq:jump-restriction}
        d_\lambda(q_\lambda^\alpha(u))=d_\alpha(u)\cap\lambda .
\end{equation}
For $\xi<\alpha$, write $\xi^\flat$ for the largest member of $\{0\}\cup\Lim$ which is not greater than $\xi$, where $\Lim$ denotes the class of limit ordinals; equivalently, $\xi=\xi^\flat+n$ for some $n<\omega$.  Thus $[\xi^\flat,\xi]$ is a finite interval.  Define
\begin{equation}\label{eq:limit-s}
\begin{split}
        s_\alpha(u)=&\ s_0(q_0^\alpha(u))
        \cup\bigcup_{\xi\in d_\alpha(u)}s_{\xi+1}(q_{\xi+1}^\alpha(u))\\
        &\cup\bigcup_{\xi\in d_\alpha(u)}\bigcup_{\zeta\in[\xi^\flat,\xi]}T_\zeta .
\end{split}
\end{equation}
This is finite.  Notice that every term is defined and lies in $X_\alpha$: by (I3), $q_{\xi+1}^\alpha(u)\in X_{\xi+1}$, and $T_\zeta\subseteq X_{\zeta+1}\subseteq X_\alpha$.

We record two bookkeeping inclusions which are the heart of the limit step.  First, if $\lambda<\alpha$ is either $0$ or a limit ordinal, then
\begin{equation}\label{eq:limit-old-s-inclusion}
        s_\lambda(q_\lambda^\alpha(u))\subseteq s_\alpha(u)
        \qquad(u\in X_\alpha).
\end{equation}
For $\lambda=0$ this is one of the terms in \eqref{eq:limit-s}.  If $\lambda$ is a non-zero limit, expand $s_\lambda(q_\lambda^\alpha(u))$ by the already constructed limit formula at stage $\lambda$ and use \eqref{eq:jump-restriction}; each resulting $s_{\eta+1}$ term and each block of $T_\zeta$'s is one of the corresponding terms in \eqref{eq:limit-s} for $s_\alpha(u)$.

Second, if $\xi<\alpha$ and $\lambda=\xi^\flat$, then
\begin{equation}\label{eq:block-expansion}
        s_{\xi+1}(q_{\xi+1}^\alpha(u))
        \subseteq
        s_\lambda(q_\lambda^\alpha(u))
        \cup\bigcup_{\zeta\in[\lambda,\xi]}T_\zeta .
\end{equation}
Indeed, iterating the successor formula \eqref{eq:successor-s} along the finite interval $[\lambda,\xi]$ gives
\[
        s_{\xi+1}(q_{\xi+1}^\alpha(u))
        \subseteq T_\xi\cup s_\xi(q_\xi^\alpha(u))
        \subseteq T_\xi\cup T_{\xi-1}\cup s_{\xi-1}(q_{\xi-1}^\alpha(u))
        \subseteq\cdots\subseteq
        s_\lambda(q_\lambda^\alpha(u))
        \cup\bigcup_{\zeta\in[\lambda,\xi]}T_\zeta,
\]
with the evident interpretation when $\lambda=\xi$.

The finite $T$-block is precisely what handles a jump inherited from only one of $u$ and $v$: the stage-$\xi+1$ data for the non-jumping element are expanded backwards through the finite successor interval $[\xi^\flat,\xi]$, while the same block is supplied directly by the jumping element.  Combining \eqref{eq:limit-old-s-inclusion} and \eqref{eq:block-expansion}, we obtain the inclusion needed for coherence: whenever $\xi\in d_\alpha(u)\cup d_\alpha(v)$,
\begin{equation}\label{eq:limit-inclusion}
        s_{\xi+1}(q_{\xi+1}^\alpha(u))
        \cup s_{\xi+1}(q_{\xi+1}^\alpha(v))
        \subseteq s_\alpha(u)\cup s_\alpha(v).
\end{equation}
For example, if $\xi\in d_\alpha(v)\setminus d_\alpha(u)$, then the first term is contained in $s_\alpha(u)\cup\bigcup_{\zeta\in[\xi^\flat,\xi]}T_\zeta$ by \eqref{eq:limit-old-s-inclusion} and \eqref{eq:block-expansion}, while the displayed block of $T_\zeta$'s is included in $s_\alpha(v)$ by \eqref{eq:limit-s}.  The other cases are direct or symmetric.

We now prove finite separation for $s_\alpha$.  Let $u,v\in X_\alpha$ and $u\wedge v=0$.  If $q_0^\alpha(u)\wedge q_0^\alpha(v)=0$, then the separating pair supplied by $s_0$ for $q_0^\alpha(u)$ and $q_0^\alpha(v)$ belongs to $s_\alpha(u)\cap s_\alpha(v)$ and dominates $u$ and $v$.  Otherwise Lemma~\ref{lem:jump-sets}(3) gives $\xi\in d_\alpha(u)\cap d_\alpha(v)$ with
\[
        q_{\xi+1}^\alpha(u)\wedge q_{\xi+1}^\alpha(v)=0.
\]
The induction hypothesis at stage $\xi+1$ gives a separating pair in
\[
        s_{\xi+1}(q_{\xi+1}^\alpha(u))
        \cap s_{\xi+1}(q_{\xi+1}^\alpha(v)),
\]
and this pair is contained in $s_\alpha(u)\cap s_\alpha(v)$ because $\xi$ is a jump of both $u$ and $v$.  Since $u\leqslant q_{\xi+1}^\alpha(u)$ and $v\leqslant q_{\xi+1}^\alpha(v)$, it separates $u$ and $v$.

It remains to prove coherence.  Let $u,v\in X_\alpha$ and $u\wedge v>0$.  By (I4), for every $\xi<\alpha$,
\[
        q_\xi^\alpha(u\wedge v)=q_\xi^\alpha(u)\wedge q_\xi^\alpha(v).
\]
Consequently,
\begin{equation}\label{eq:jumps-subset}
        d_\alpha(u\wedge v)\subseteq d_\alpha(u)\cup d_\alpha(v).
\end{equation}
Indeed, if $\xi\notin d_\alpha(u)\cup d_\alpha(v)$, then
\[
q_\xi^\alpha(u\wedge v)
=q_\xi^\alpha(u)\wedge q_\xi^\alpha(v)
=q_{\xi+1}^\alpha(u)\wedge q_{\xi+1}^\alpha(v)
=q_{\xi+1}^\alpha(u\wedge v).
\]
The first term in \eqref{eq:limit-s} for $s_\alpha(u\wedge v)$ is controlled by the coherence of $s_0$:
\[
        s_0(q_0^\alpha(u\wedge v))
        \subseteq
        \bigl(s_0(q_0^\alpha(u))\cup s_0(q_0^\alpha(v))\bigr)^\wedge
        \subseteq (s_\alpha(u)\cup s_\alpha(v))^\wedge .
\]
Now let $\xi\in d_\alpha(u\wedge v)$.  Put
\[
        u_\xi=q_{\xi+1}^\alpha(u),
        \qquad v_\xi=q_{\xi+1}^\alpha(v).
\]
Then $u_\xi\wedge v_\xi=q_{\xi+1}^\alpha(u\wedge v)>0$, and by the induction hypothesis at stage $\xi+1$,
\[
        s_{\xi+1}(u_\xi\wedge v_\xi)
        \subseteq\bigl(s_{\xi+1}(u_\xi)\cup s_{\xi+1}(v_\xi)\bigr)^\wedge .
\]
By \eqref{eq:jumps-subset}, $\xi\in d_\alpha(u)\cup d_\alpha(v)$, so \eqref{eq:limit-inclusion} puts the right-hand side inside $(s_\alpha(u)\cup s_\alpha(v))^\wedge$.  Finally, the $T$-block attached to such a $\xi$ in \eqref{eq:limit-s} is contained in $s_\alpha(u)\cup s_\alpha(v)$ by \eqref{eq:jumps-subset}.  Thus every term in $s_\alpha(u\wedge v)$ belongs to $(s_\alpha(u)\cup s_\alpha(v))^\wedge$, proving coherence.

The induction is complete.  Taking $\alpha=\kappa$ gives a meet-closed base $X_\kappa$ of $\bfA_\kappa=\bfA$ and a coherent finite-separation map $s_\kappa$ on it.
\end{proof}

\begin{theorem}\label{thm:projective-main}
A Boolean algebra is projective if and only if it has a base closed under non-zero meets which carries a coherent finite-separation map.
\end{theorem}

\begin{proof}
The forward implication is Proposition~\ref{prop:projective-to-base}.  The reverse implication is Proposition~\ref{prop:base-to-projective}.
\end{proof}

We say that a Boolean algebra $\bfA$ is \emph{rc-filtered} if there is a closed unbounded family of countable relatively complete subalgebras of $\bfA$.

The following theorem is well known (see \cite[Corollary 2.2.7, Theorem 1.3.2]{HeindorfShapiro}); we include the standard construction because it will be used below.

\begin{theorem}\label{thm:small-rc-chain}
If a Boolean algebra $\bfA$ is rc-filtered and $|\bfA|\leqslant\omega_1$, then there is a continuous chain $(\bfA_\alpha)_{\alpha\leqslant\omega_1}$ of subalgebras of $\bfA$ such that
\[
        \bfA_{\omega_1}=\bfA,
\]
$\bfA_0$ is countable, $\bfA_\alpha\rc\bfA_{\alpha+1}$ for every $\alpha<\omega_1$, and $\bfA_{\alpha+1}$ is countably generated over $\bfA_\alpha$.
\end{theorem}

\begin{proof}
Let $\calR$ be a club of countable relatively complete subalgebras of $\bfA$, and enumerate $\bfA$ as $(a_\alpha)_{\alpha<\omega_1}$, allowing repetitions when $|\bfA|<\omega_1$.  Choose $\bfA_0\in\calR$.  Suppose that $\bfA_\alpha\in\calR$ has been defined.  Choose $\bfA_{\alpha+1}\in\calR$ such that
\[
        \bfA_\alpha\cup\{a_\alpha\}\subseteq\bfA_{\alpha+1};
\]
this is possible because $\calR$ is cofinal in $[\bfA]^\omega$.  At a non-zero limit ordinal $\lambda<\omega_1$, put
\[
        \bfA_\lambda=\bigcup_{\alpha<\lambda}\bfA_\alpha.
\]
The chain below $\lambda$ is countable, so the closure of $\calR$ gives $\bfA_\lambda\in\calR$.

The chain below $\omega_1$ is continuous and has union $\bfA$.  Put $\bfA_{\omega_1}=\bfA$; this makes the chain continuous also at its endpoint.  Moreover, if $\alpha<\beta\leqslant\omega_1$, then $\bfA_\alpha\rc\bfA$ and $\bfA_\alpha\leqslant\bfA_\beta$, hence $\bfA_\alpha\rc\bfA_\beta$.  Finally, each successor algebra below the endpoint is countable and therefore countably generated over its predecessor.
\end{proof}

By \cite[Theorem 2.2.3]{HeindorfShapiro} and Proposition~\ref{prop:fn-fns-equivalence}, a Boolean algebra is rc-filtered if and only if it has global $(\fns)$.  A coherent finite-separation map on a meet-closed base implies projectivity by Theorem~\ref{thm:projective-main}, and hence implies global $(\fns)$.  At cardinality at most $\omega_1$, the converse also holds.

\begin{theorem}\label{thm:small-coherent-base}
Let $\bfA$ be a Boolean algebra with $|\bfA|\leqslant\omega_1$.  If $\bfA$ has global $(\fns)$, then $\bfA$ has a base $X\subseteq\pos{\bfA}$, closed under non-zero meets, which carries a coherent finite-separation map.
\end{theorem}

\begin{proof}
If $\bfA$ is countable, then it is projective by Halmos's theorem, and Proposition~\ref{prop:projective-to-base} applies.  Assume therefore that $|\bfA|=\omega_1$, and let
\[
        s\colon\pos{\bfA}\longrightarrow[\pos{\bfA}]^{<\omega}
\]
witness global $(\fns)$.  Closing a countable subset of $\bfA$ successively under the Boolean operations and under $s$ produces a countable subalgebra $\bfB\leqslant\bfA$ such that
\[
        s(b)\subseteq\bfB\qquad(b\in\pos{\bfB}).
\]
Consequently, the family
\[
\calR=\{\bfB\leqslant\bfA:\ |\bfB|\leqslant\aleph_0
        \text{ and }s(b)\subseteq\bfB\text{ for every }b\in\pos{\bfB}\}
\]
is a club.  For every $\bfB\in\calR$, apply Lemma~\ref{lem:closed-subalgebra} with $X=\pos{\bfA}$ and $P=\pos{\bfB}$.  Since these are bases of $\bfA$ and $\bfB$, respectively, the lemma yields $\bfB\rc\bfA$.  Thus $\bfA$ is rc-filtered.

Theorem~\ref{thm:small-rc-chain} now gives a continuous chain $(\bfA_\alpha)_{\alpha\leqslant\omega_1}$ with $\bfA_{\omega_1}=\bfA$, satisfying condition~(4) of Theorem~\ref{thm:projective-known}; hence $\bfA$ is projective.  Proposition~\ref{prop:projective-to-base} completes the proof.
\end{proof}

As shown by Heindorf and Shapiro \cite[Proposition 6.3.2]{HeindorfShapiro}, projective and rc-filtered Boolean algebras do not coincide in general.  At cardinality at most $\omega_1$, they do coincide.

\begin{corollary}[{\cite[Corollary 2.2.7]{HeindorfShapiro}}]\label{cor:small-projective}
For Boolean algebras of cardinality at most $\omega_1$, projectivity, $(\sfn)$, rc-filteredness, global $(\fns)$ (equivalently $(\fn)$), and the existence of a meet-closed base carrying a coherent finite-separation map are equivalent.
\end{corollary}

\begin{proof}
The equivalence of rc-filteredness and global $(\fns)$ is \cite[Theorem~2.2.3]{HeindorfShapiro} together with Proposition~\ref{prop:fn-fns-equivalence}.  Theorem~\ref{thm:small-coherent-base} and Theorem~\ref{thm:projective-main} give the equivalence with the coherent base condition.  Projectivity implies $(\sfn)$, and $(\sfn)$ implies $(\fn)$ \cite{HeindorfShapiro,MilovichSFN}; the converse implications at cardinality at most $\omega_1$ now follow from the already established equivalence of $(\fn)$ and projectivity.  Finally, projective Boolean algebras are rc-filtered, while the converse at cardinality at most $\omega_1$ follows from Theorem~\ref{thm:small-rc-chain} and condition~(4) of Theorem~\ref{thm:projective-known}.
\end{proof}

\section{Weakly projective Boolean algebras}\label{sec:weakly-projective}

The weakly projective theorem follows the same pattern, with regular subalgebras and $\pi$-bases replacing relatively complete subalgebras and bases.

\begin{proposition}\label{prop:pibase-to-weak}
Let $\bfA$ be a Boolean algebra.  Suppose that $X\subseteq\pos{\bfA}$ is a $\pi$-base closed under non-zero meets and that $s\colon X\to[X]^{<\omega}$ witnesses $(\fns)^*$ on $X$.  Then $\bfA$ is weakly projective.
\end{proposition}

\begin{proof}
For every countable $D\subseteq\bfA$ we use the same least iterated closure as in Proposition~\ref{prop:base-to-projective}.  Put $\bfD_0=\langle D\rangle$ and $P_0=\emptyset$.  Having constructed $\bfD_n$ and $P_n$, choose, for every $b\in\pos{\bfD_n}$, an element $x_b\in X$ with $x_b\leqslant b$, and put
\[
        P_{n+1}=\operatorname{cl}_{s,\wedge}
        \bigl(P_n\cup\{x_b:b\in\pos{\bfD_n}\}\bigr),
        \qquad
        \bfD_{n+1}=\langle\bfD_n\cup P_{n+1}\rangle.
\]
Let $P_\bfD=\bigcup_{n<\omega}P_n$ and $\bfD=\bigcup_{n<\omega}\bfD_n$.  Then $\bfD$ is countable and contains $D$; every positive element of $\bfD$ contains a member of $P_\bfD$, so $P_\bfD$ is a $\pi$-base of $\bfD$; and the least-closure construction gives $s(p)\subseteq P_\bfD$ for every $p\in P_\bfD$.

Let $\calR$ be the family of all countable subalgebras with such a $\pi$-base.  Then $\bfA=\bigcup\calR$.  By Lemma~\ref{lem:closed-subalgebra}, each member of $\calR$ is regular in $\bfA$.

We show that $\calR$ is a countable regular-generating family.  Fix an arbitrary subfamily $\calC\subseteq\calR$ and put
\[
        \bfE=\Bigl\langle\bigcup_{\bfD\in\calC}\bfD\Bigr\rangle .
\]
If $\calC=\emptyset$, then $\bfE=\{0,1\}\reg\bfA$, so assume that $\calC\neq\emptyset$.  For each $\bfD\in\calC$ choose a witnessing $\pi$-base $P_\bfD$.  Let $P$ be the set of all non-zero finite meets
\[
        p_0\wedge\cdots\wedge p_{n-1},
        \qquad p_i\in P_{\bfD_i},\ \bfD_i\in\calC .
\]
Then $P\subseteq X\cap\bfE$.  We verify that $P$ is a $\pi$-base of $\bfE$.

Let $0<e\in\bfE$.  The element $e$ belongs to the subalgebra generated by finitely many members $\bfD_0,
\ldots,\bfD_{n-1}$ of $\calC$.  By distributivity, $e$ contains a non-zero element of the form
\[
        d_0\wedge\cdots\wedge d_{n-1},
        \qquad d_i\in\pos{\bfD_i}.
\]
We refine this meet one coordinate at a time.  Suppose that $p_0\in P_{\bfD_0},\ldots,p_{i-1}\in P_{\bfD_{i-1}}$ have been chosen so that
\[
        b_i=p_0\wedge\cdots\wedge p_{i-1}\wedge d_i\wedge\cdots\wedge d_{n-1}>0
\]
(with the evident interpretation for $i=0$).  Since $\bfD_i\reg\bfA$, Lemma~\ref{lem:regular-criterion} gives $a_i\in\pos{\bfD_i}$ such that every non-zero member of $\bfD_i$ below $a_i$ meets $b_i$.  The element $a_i\wedge d_i$ is non-zero, for otherwise $a_i$ itself would be a non-zero member of $\bfD_i$ below $a_i$ disjoint from $b_i$.  Choose $p_i\in P_{\bfD_i}$ with $p_i\leqslant a_i\wedge d_i$.  Then $p_i\wedge b_i>0$, so the induction continues.  At the end we have a member $p_0\wedge\cdots\wedge p_{n-1}$ of $P$ below $e$.  Hence $P$ is a $\pi$-base of $\bfE$.

If $p=p_0\wedge\cdots\wedge p_{n-1}\in P$, Lemma~\ref{lem:finite-meet-star} gives
\[
        s(p)\subseteq (s(p_0)\cup\cdots\cup s(p_{n-1}))^{\wedge}\subseteq\bfE,
\]
because each $s(p_i)$ is contained in the corresponding $\bfD_i$.  Lemma~\ref{lem:closed-subalgebra} therefore yields $\bfE\reg\bfA$.  Thus $\calR$ is a countable regular-generating family, and Theorem~\ref{thm:weak-known}(3) implies that $\bfA$ is weakly projective.
\end{proof}

\begin{proposition}\label{prop:weak-to-pibase}
Every weakly projective Boolean algebra has a $\pi$-base closed under non-zero meets which carries a coherent finite-separation map.
\end{proposition}

\begin{proof}
By Theorem~\ref{thm:weak-known}, a weakly projective algebra $\bfA$ contains a dense projective subalgebra $\bfB$.  By Theorem~\ref{thm:projective-main}, $\bfB$ has a base $X$ closed under non-zero meets and carrying a coherent finite-separation map.  Since $\bfB$ is dense in $\bfA$, the same set $X$ is a $\pi$-base of $\bfA$.
\end{proof}

\begin{theorem}\label{thm:weak-main}
A Boolean algebra is weakly projective if and only if it has a $\pi$-base closed under non-zero meets which carries a coherent finite-separation map.
\end{theorem}

\begin{proof}
This is the combination of Propositions~\ref{prop:pibase-to-weak} and \ref{prop:weak-to-pibase}.
\end{proof}

Recall that a Boolean algebra is a \emph{Cohen algebra} if its completion is the completion of a free Boolean algebra.  Equivalently, it is co-complete with a free Boolean algebra.  A Boolean algebra has \emph{uniform density} $\kappa$ if every non-zero relative algebra has density $\kappa$.

\begin{corollary}\label{cor:cohen}
Let $\kappa$ be an infinite cardinal and let $\bfA$ be a Boolean algebra of uniform density $\kappa$.  Then $\bfA$ is a Cohen algebra if and only if it has a $\pi$-base closed under non-zero meets which carries a coherent finite-separation map.
\end{corollary}

\begin{proof}
Heindorf--Shapiro's form of Koppelberg's Cohen-algebra characterisation says that a Boolean algebra is Cohen if and only if it is weakly projective and has uniform density; see \cite[Proposition 5.2.5]{HeindorfShapiro} and \cite{KoppelbergCohen}.  Now apply Theorem~\ref{thm:weak-main}.
\end{proof}

\section{The global coherent finite-separation property}\label{sec:global}

This section proves that the coherent property becomes countability when it is imposed on every non-zero element.

We begin with a finite-support normalisation inside free Boolean algebras.  If $F=\Fr(Y)$ and $u\in F$, a finite set $S\subseteq Y$ is a \emph{support} of $u$ if $u\in\langle S\rangle$.

\begin{lemma}\label{lem:support-independence}
Let $Y$ be an independent generating set for $F=\Fr(Y)$.  Let $R,S_0,\ldots,S_{n-1}$ be finite subsets of $Y$ such that $R,S_0,\ldots,S_{n-1}$ are pairwise disjoint.  If $0<u_i\in\langle S_i\rangle$ for each $i<n$ and $0<v\in\langle R\rangle$, then
\[
        v\wedge u_0\wedge\cdots\wedge u_{n-1}>0 .
\]
In particular, non-zero elements with pairwise disjoint finite supports have non-zero meet.
\end{lemma}

\begin{proof}
For each $i<n$, choose a valuation of $S_i$ which makes $u_i$ equal to $1$, and choose a valuation of $R$ which makes $v$ equal to $1$.  Since the supports are pairwise disjoint, these finite valuations are compatible and extend to a valuation of $Y$.  The corresponding elementary product is a non-zero element below $v\wedge u_0\wedge\cdots\wedge u_{n-1}$.
\end{proof}

\begin{lemma}\label{lem:delta-pattern}
Let $B$ be an uncountable subalgebra of a free Boolean algebra $F=\Fr(Y)$.  Then there are an uncountable set $I$, pairwise distinct elements $b_\xi\in B\setminus\{0,1\}$ for $\xi\in I$, a finite root $R\subseteq Y$, pairwise disjoint non-empty finite sets $P_\xi\subseteq Y\setminus R$, and, for each atom $\tau\in\At(\langle R\rangle)$, a Boolean term $t_\tau$ in $n$ variables such that, after identifying $P_\xi$ with these $n$ variables,
\begin{equation}\label{eq:delta-pattern}
        b_\xi=\bigvee_{\tau\in\At(\langle R\rangle)}
        \bigl(\tau\wedge t_\tau(P_\xi)\bigr)
        \qquad(\xi\in I).
\end{equation}
Moreover, for some atom $\tau_0\in\At(\langle R\rangle)$,
\[
        0<t_{\tau_0}<1
\]
in the free Boolean algebra on $n$ generators.  Consequently, for every finite non-empty $H\subseteq I$,
\[
        0<\bigwedge_{\xi\in H}b_\xi<1 .
\]
\end{lemma}

\begin{proof}
Choose pairwise distinct elements $c_\xi\in B\setminus\{0,1\}$, $\xi<\omega_1$, and finite supports $S_\xi\subseteq Y$ with $c_\xi\in\langle S_\xi\rangle$.  The family $\{S_\xi:\xi<\omega_1\}$ is uncountable.  Indeed, if only countably many supports occurred, then the union of the corresponding finite subalgebras would be countable, contradicting the choice of the pairwise distinct elements $c_\xi$.  Apply the $\Delta$-system lemma and then thin in three successive, simultaneous steps.  First fix a finite root $R$ so that the supports form a $\Delta$-system; next fix a common positive size $n$ of the private parts $P_\xi=S_\xi\setminus R$ and choose an enumeration of each $P_\xi$ by the same set of $n$ variable positions; finally, because only finitely many Boolean functions occur on the fixed finite set of root and private variables, thin once more so that all $c_\xi$ are represented by one and the same Boolean term.  The private parts are pairwise disjoint.  Their common size is positive: if $n=0$, then all surviving elements would lie in the finite algebra $\langle R\rangle$, contradicting their pairwise distinctness.

Writing this fixed term over the atoms of the finite algebra $\langle R\rangle$ gives terms $t_\tau$ with
\[
        c_\xi=\bigvee_{\tau\in\At(\langle R\rangle)}
        \bigl(\tau\wedge t_\tau(P_\xi)\bigr)
        \qquad(\xi<\omega_1).
\]
Put $b_\xi=c_\xi$ and pass to the resulting uncountable index set.  Since the $b_\xi$ are pairwise distinct while the pattern over $R$ is fixed, at least one coefficient $t_{\tau_0}$ is neither $0$ nor $1$; otherwise the displayed expression would be independent of $\xi$.  This coefficient is non-zero, and therefore, if $H\subseteq I$ is finite and non-empty, then
\[
        \tau_0\wedge\bigwedge_{\xi\in H}t_{\tau_0}(P_\xi)>0
\]
by Lemma~\ref{lem:support-independence}.  Hence $\bigwedge_{\xi\in H}b_\xi>0$.  The same meet is below any one $b_\xi$, and $b_\xi<1$, so it is $<1$.
\end{proof}

\begin{theorem}\label{thm:global-countable}
A Boolean algebra has the global $(\fns)^*$ property if and only if it is countable.
\end{theorem}

\begin{proof}
First suppose that $\bfA$ is countable.  Write $\bfA$ as an increasing union of finite subalgebras
\[
        \bfA_0\subseteq\bfA_1\subseteq\cdots,
        \qquad \bfA_0=\{0,1\}.
\]
For $a\in\pos{\bfA}$ let $\rho(a)$ be the least $n$ such that $a\in\bfA_n$, and define
\[
        s(a)=\pos{\bfA_{\rho(a)}}.
\]
If $a,b\in\pos{\bfA}$ are disjoint and, say, $\rho(a)\leqslant\rho(b)$, then $a$ and $-a$ belong to $s(a)\cap s(b)$, they are disjoint, $a\leqslant a$, and $b\leqslant -a$.  Thus finite separation holds.  If $a\wedge b>0$ and $m=\max\{\rho(a),\rho(b)\}$, then $s(a)\cup s(b)=\pos{\bfA_m}$: the finite algebras are increasing, and one of $a,b$ has rank exactly $m$.  Also $a\wedge b\in\bfA_m$, whence
\[
        s(a\wedge b)\subseteq\pos{\bfA_m}\subseteq(s(a)\cup s(b))^{\wedge}.
\]
So $\bfA$ has global $(\fns)^*$.

Conversely, suppose that $\bfA$ is uncountable and that $s\colon\pos{\bfA}\to[\pos{\bfA}]^{<\omega}$ witnesses global $(\fns)^*$.  Since $\pos{\bfA}$ is a base of $\bfA$ closed under non-zero meets, Theorem~\ref{thm:projective-main} implies that $\bfA$ is projective.  By Theorem~\ref{thm:projective-known}, we may regard $\bfA$ as a subalgebra of a free Boolean algebra $F=\Fr(Y)$.

Apply Lemma~\ref{lem:delta-pattern}.  We obtain an uncountable family $(b_\xi)_{\xi\in I}$, a finite root $R$, pairwise disjoint private supports $P_\xi$, fixed coefficients $t_\tau$, and an atom $\tau_0\in\At(\langle R\rangle)$ such that $0<t_{\tau_0}<1$ and \eqref{eq:delta-pattern} holds.  For every $u\in F^+$ choose a finite support $S_u\subseteq Y$ with $u\in\langle S_u\rangle$.  For $\xi\in I$ put
\[
        E_\xi=\Bigl\{\eta\in I:
        P_\eta\cap\bigcup_{u\in s(b_\xi)}S_u\neq\emptyset\Bigr\}.
\]
Each $E_\xi$ is finite, because $s(b_\xi)$ is finite, each $S_u$ is finite, and the sets $P_\eta$ are pairwise disjoint and non-empty.  Thinning $I$ if necessary, assume that $|E_\xi|=m$ for all $\xi\in I$.  Choose distinct
\[
        \alpha_0,\ldots,\alpha_{m+1}\in I
\]
and put
\[
        H=\{\alpha_0,\ldots,\alpha_{m+1}\},
        \qquad
        c=\bigwedge_{\alpha\in H}b_\alpha,
        \qquad h=-c.
\]
By Lemma~\ref{lem:delta-pattern}, $0<c<1$, so $h\in\pos{\bfA}$.

Apply finite separation to the disjoint pair $c,h$.  There are disjoint $p,q\in s(c)\cap s(h)$ with $c\leqslant p$ and $h\leqslant q$.  Since $p\wedge q=0$,
\[
        p\leqslant -q\leqslant -h=c,
\]
and hence $p=c$ and $q=h$.  In particular, $h\in s(c)$.

By Lemma~\ref{lem:finite-meet-star}, $h$ is a non-zero finite meet of elements from
\[
        \bigcup_{\alpha\in H}s(b_\alpha).
\]
Choose a finite non-empty set $W\subseteq\bigcup_{\alpha\in H}s(b_\alpha)$ such that
\[
        h=\bigwedge W .
\]
Here $W$ is allowed to be a singleton: by definition, $E^\wedge$ contains all non-zero meets of non-empty finite subsets, including one-element meets.
Since $0<t_{\tau_0}<1$ and the private supports $P_\alpha$ are pairwise disjoint,
\[
        \tau_0\wedge c
        =\tau_0\wedge\bigwedge_{\alpha\in H}t_{\tau_0}(P_\alpha)>0
\]
by Lemma~\ref{lem:support-independence}.  As $c=-h=\bigvee_{w\in W}-w$, there is $w\in W$ such that
\[
        \tau_0\wedge -w>0 .
\]
Choose $\alpha^*\in H$ with $w\in s(b_{\alpha^*})$.  Since $|E_{\alpha^*}|=m$ and $|H|=m+2$, there is
\[
        \beta\in H\setminus(E_{\alpha^*}\cup\{\alpha^*\}).
\]
Then $P_\beta\cap S_w=\emptyset$.  By construction $P_\beta\cap R=\emptyset$ as well.  Consequently
\[
        P_\beta\cap(R\cup S_w)=\emptyset,
\]
where $R\cup S_w$ supports $\tau_0\wedge-w$, whereas $P_\beta$ supports $-t_{\tau_0}(P_\beta)$.  As $0<t_{\tau_0}<1$, the latter element is non-zero, and Lemma~\ref{lem:support-independence} gives
\[
        d=(\tau_0\wedge -w)\wedge -t_{\tau_0}(P_\beta)>0 .
\]
Restricting the normal form \eqref{eq:delta-pattern} for $b_\beta$ to the root atom $\tau_0$ gives
\[
        \tau_0\wedge b_\beta
        =\tau_0\wedge t_{\tau_0}(P_\beta).
\]
Since $d\leqslant\tau_0\wedge-t_{\tau_0}(P_\beta)$, we have $d\wedge b_\beta=0$, and hence
\[
        d\leqslant-b_\beta\leqslant h.
\]
But $d\leqslant -w$, while $h\leqslant w$ because $h=\bigwedge W$.  This contradiction proves that no uncountable Boolean algebra has global $(\fns)^*$.
\end{proof}

\begin{remark}\label{rem:global-local}
Theorem~\ref{thm:global-countable} shows that the local nature of Theorems~\ref{thm:projective-main} and \ref{thm:weak-main} is not cosmetic.  In an uncountable projective Boolean algebra, the coherent finite-separation map must be placed on a carefully chosen base, not on the whole positive cone.
\end{remark}

\section{Stone-dual consequences and further directions}\label{sec:further}

We first record the Stone-dual form of the main results.  Let $K$ be a Stone space.  A family
\[
        \calX\subseteq\Clop(K)\setminus\{\emptyset\}
\]
is a \emph{clopen decomposition base} if every non-empty clopen subset of $K$ is a finite disjoint union of members of $\calX$; it is a \emph{clopen $\pi$-base} if every non-empty open subset of $K$ contains a member of $\calX$.  A finite-separation map, and its coherent version, are defined by translating Definition~\ref{def:fns-star}: replace $\leqslant$, $\wedge$, and $0$ by $\subseteq$, $\cap$, and $\emptyset$, respectively.

\begin{corollary}[Stone-dual form]\label{cor:stone-dual}
Let $K$ be a Stone space.
\begin{enumerate}
\item Within the class of Stone spaces, $K$ is a Dugundji compactum, equivalently an AE$(0)$ compactum, equivalently a retract of a Cantor cube, if and only if $K$ has a clopen decomposition base, closed under non-empty intersections, which carries a coherent finite-separation map.
\item $K$ is openly generated if and only if $\Clop(K)\setminus\{\emptyset\}$ carries a finite-separation map.
\item $K$ is co-absolute with a Dugundji compactum if and only if $K$ has a clopen $\pi$-base, closed under non-empty intersections, which carries a coherent finite-separation map.
\item $K$ is metrisable if and only if $\Clop(K)\setminus\{\emptyset\}$ carries a coherent finite-separation map.
\end{enumerate}
\end{corollary}

\begin{proof}
Put $\bfA=\Clop(K)$.  Under Stone duality, non-zero meets become non-empty intersections, algebraic bases become clopen decomposition bases, and algebraic $\pi$-bases become clopen $\pi$-bases.  A free Boolean algebra on $\kappa$ generators has Stone space $2^\kappa$, and retractions dualise to retractions.  Thus the first assertion follows from Theorem~\ref{thm:projective-main}.  Within the class of Stone spaces, Haydon's equivalence of Dugundji compacta and AE$(0)$ compacta \cite[Theorems~1 and~3]{Haydon}, together with the standard embedding of every Stone space in a Cantor cube, identifies these spaces exactly with retracts of Cantor cubes; see also the inverse-spectral background in Shchepin \cite[Sections~1--2]{Shchepin} and the algebraic formulation in Heindorf--Shapiro \cite[Theorem~1.3.2]{HeindorfShapiro}.  The second assertion follows from Proposition~\ref{prop:fn-fns-equivalence} and the Stone-dual characterisation of the Freese--Nation property by openly generated compacta; see \cite[Theorem~2.2.3]{HeindorfShapiro}.

The completion of $\Clop(K)$ is $\RO(K)$, and two compacta are co-absolute precisely when their regular-open algebras are isomorphic.  Hence the third assertion is the Stone dual of Theorem~\ref{thm:weak-main}.  Finally, Theorem~\ref{thm:global-countable} says that the coherent map exists on all non-empty clopen sets exactly when $\Clop(K)$ is countable.  For a Stone space this is equivalent to second countability, and hence to metrisability.
\end{proof}

We next record a boundary of the present method outside Boolean algebras.

\begin{remark}[Orthomodular lattices]\label{rem:oml}
For background on orthomodular lattices and their use as quantum logics we refer to Kalmbach~\cite{Kalmbach} and Pt\'ak--Pulmannov\'a~\cite{PtP}.  The order-theoretic definition of $(\fn)$ makes sense for every lattice.  In an ortholattice, and hence in an orthomodular lattice, there is also a natural separation form: for orthogonal pairs $a\ortho b$, meaning $a\leqslant b^\perp$, ask for $c,d\in s(a)\cap s(b)$ with $c\ortho d$, $a\leqslant c$, and $b\leqslant d$.  With this orthogonality convention, global $(\fn)$ and global $(\fns)$ remain equivalent; the proof of Proposition~\ref{prop:fn-fns-equivalence} uses only the order and the orthocomplementation.  If one separates all meet-zero pairs instead, the equivalence is no longer formal, because $a\wedge b=0$ need not imply $a\leqslant b^\perp$.

The coherent strengthening can also be written down, replacing Boolean complements by orthocomplements and finite meets by lattice meets.  The Boolean characterisation itself, however, does not transfer verbatim.  The proofs above use finite disjoint decompositions over Boolean bases, distributive refinement, relatively complete Boolean subalgebras and their projection maps, Stone duality, and the independence of finite supports in free Boolean algebras.  A particularly simple obstruction is that the global countability theorem is false for orthomodular lattices.  Let $I$ be uncountable and let $L_I$ be the horizontal sum of copies of the four-element Boolean algebra; write the two proper elements in the $i$th block as $p_i$ and $q_i=p_i^\perp$.  Then $L_I$ is an uncountable orthomodular lattice, and
\[
        s(1)=\{1\},\qquad s(p_i)=s(q_i)=\{p_i,q_i\}
        \quad(i\in I)
\]
witnesses global $(\fns)^*$ in the orthogonality sense.  The only non-trivial orthogonal pairs are $p_i,q_i$, and the coherence condition is immediate because non-zero meets occur only inside a single four-element block or with $1$.

Thus the part of the theory that should pass is the local, blockwise Boolean part.  A useful orthomodular analogue would need additional hypotheses ensuring enough compatible or central finite refinements; otherwise coherent finite separation measures only the Boolean pieces and misses the genuinely quantum behaviour.
\end{remark}
\begin{problem}\label{prob:oml}
Find an orthomodular-lattice analogue of Theorems~\ref{thm:projective-main} and~\ref{thm:weak-main}, using the standard block and compatibility theory of orthomodular lattices as in \cite{Kalmbach,PtP}.  In particular, decide whether coherent finite separation should be imposed for meet-zero pairs, for orthogonal pairs, or only inside Boolean blocks, and identify the additional hypotheses under which projectivity can be recovered from such finite data.
\end{problem}

\begin{problem}\label{prob:skeleton-compatible}
Can the base in Theorem~\ref{thm:projective-main} always be chosen canonically from a fixed additive relatively complete skeleton in the standard sense of Definition~\ref{def:skeleton-terminology}, in such a way that it is closed under all skeleton projections $q_{\bfC}^{\bfA}$?
\end{problem}

A positive answer would turn the present characterisation into a more invariant tool: the witnessing base would not merely exist, but would reflect a chosen projective resolution of the algebra.

\begin{problem}\label{prob:bounded-witnesses}
For $n<\omega$, classify those Boolean algebras admitting a base or $\pi$-base $X$ and a coherent finite-separation map $s$ with $|s(x)|\leqslant n$ for every $x\in X$.
\end{problem}

Even the first few finite bounds should distinguish meaningful subclasses of projective or weakly projective algebras.  The countable construction in Theorem~\ref{thm:global-countable} gives finite witnesses, but with no uniform bound in general.

\section*{Acknowledgements}
The first-named author gratefully acknowledges support received from
NCN Sonata-Bis~13 (2023/50/E/ST1/00067).  The second- and third-named authors would like to thank W.~Bielas, W.~Kubi\'{s}, and the participants of the Topology and Set Theory Seminar in Katowice for helpful discussions and valuable comments.

\end{document}